\documentclass[a4paper,12pt]{article}
\usepackage[utf8]{inputenc}
\usepackage{amsmath,amssymb,amsthm}
\usepackage{hyperref}

\makeatletter
\def\@seccntformat#1{\csname the#1\endcsname.\ } 
\makeatother

\newtheorem{lemma}{Lemma}
\newtheorem{proposition}{Proposition}
\newtheorem{corollary}{Corollary}

\theoremstyle{remark}

\def\FF{\mathbb{F}}
\def\GG{\mathbb{G}}
\def\Cal{\mathrm{Cal}}

\title{On the coset graph construction\\
       of distance-regular graphs%
\thanks{The work of M. J. Shi is supported by the National Natural Science Foundation of China (12071001) and the Excellent Youth Foundation of Natural Science Foundation of Anhui Province (1808085J20); the work of D. S. Krotov is supported
within the framework of the state contract of the Sobolev Institute of Mathematics
(FWNF-2022-0017).
}
}
\author{Minjia Shi%
\thanks{Key Laboratory of Intelligent Computing and Signal Processing, Ministry of Education, School of Mathematical Sciences, Anhui University, Hefei, 230601, China, {\tt smjwcl.good@163.com}}%
,
Denis S. Krotov%
\thanks{Sobolev Institute of Mathematics, Novosibirsk 630090, Russia, {\tt krotov@math.nsc.ru}}%
,
Patrick Sol\'e%
\thanks{I2M (CNRS, Centrale Marseille, University of Aix-Marseille), Marseilles, France, {\tt sole@enst.fr}}%
}

\date{}

\begin{document}
\maketitle

\begin{abstract}
 We show that no more new distance-regular graphs in the tables of the book of (Brouwer, Cohen, Neumaier, 1989) can be produced by using the coset graph of additive completely regular codes over finite fields.
\end{abstract}


{\bf AMS Classification (MSC 2010):} Primary  94B05 Secondary 05E30
\section{Introduction}

Since the times of Delsarte \cite{Delsarte:1973}, an important tool to construct distance-regular graphs has been the coset graph of a completely regular code. Many examples can be found
in Chapter~11 of \cite{Brouwer}, involving notably the Golay codes and the Kasami codes (for a new description of these codes see e.g.~\cite{ShiKroSol:Kasami}). More recently,
a new distance-regular graph related to the dodecacode, an important additive $\FF_4$-code of length~$12$ was constructed in~\cite{SKS:drg}.
This solved both a thirty year old problem in graph theory
\cite[Ch.\,14]{Brouwer}, and a forty-five year old problem in coding theory \cite{BZZ:1974:UPC}.

In the present note, we find out which graphs with the parameters in the tables of \cite[Ch.\,14]{Brouwer} can be realized as coset graphs of completely regular additive codes over a finite field.
A trivial necessary condition is that the number of vertices should be a power of a prime. More elaborate necessary conditions include a refinement of the Lloyd theorem for equitable partitions (Lemma~\ref{l:wd}),
and a divisibility condition bearing on the entries of the intersection array (Corollary~\ref{c:fact}). When all these techniques are exhausted, we can eliminate the remaining parameters by showing that the needed
code does not exist, either by a direct combinatorial argument, or by long electronic computations using the dedicated software~\cite{Kurz:LINCODE} or~\cite{QextNewEdition}. We conclude that no more additive
codes are left to be found satisfying the said condition.

The material is arranged as follows. The next section lists the parameters we are considering. Section~\ref{s:def} collects basic facts and notations. Section~\ref{s:wd} gives the two theoretical tools mentioned above.
Sections~\ref{s:42} to~\ref{s:140} study the codes remaining after use of these two tools. Section~\ref{s:concl} concludes the article.

\section{Parameters}\label{s:par}
The only ``unsolved'' parameters of distance-regular graphs with prime-power order
in the [BCN] table \cite[Ch.\,14]{Brouwer} are the following (the most of the cited updates are mentioned in~\cite[Ch.\,``Tables'']{vDKT:DR}):
\begin{itemize}
 \item $v=2187=3^7$, $\{140,126,20,1;1,10,126,140\}$ (Section~\ref{s:140});

 \item $v=243=3^5$, $\{22,16,5;\underline{1,2,20}\}$; a graph
 does not exist \cite{SumVor:2016};
 \item $v=256=2^8$, $\{25,24,3;\underline{1,3},20\}$ (Corollary~\ref{c:fact});

 \item $v=729=3^6$, $\{26,24,19;\underline{1,3},8\}$ (Corollary~\ref{c:fact});

 \item $v=1024=2^{10}$, $\{31,30,17;1,2,15\}$; a graph exists
 (Kasami graph \cite[Thm. 11.2.1 $(13)$]{Brouwer}, $q=j=2$).

 \item $v=625=5^4$, $\{32,28,9;\underline{1,2,28}\}$ (Corollary~\ref{c:fact});

 \item $v=1024=2^{10}$, $\{33,30,15;1,2,15\}$; a graph exists \cite{SKS:drg};

 \item $v=625=5^4$, $\{ 36,28,4;1,2,24 \}$ (Lemma~\ref{l:wd});

 \item $v=343=7^3$, $\{ 42,30,12;1,6,28 \}$ (Section~\ref{s:42});

 \item $v=243=3^5$, $\{ 44,36,5;\underline{1,9},40 \}$ (Corollary~\ref{c:fact});

 \item $v=343=7^3$, $\{ 48,35,9;\underline{1,7},40 \}$ (Corollary~\ref{c:fact});

 \item $v=625=5^4$, $\{ 52,42,16;1,6,28 \}$ (Section~\ref{s:52});

 \item $v=729=3^6$, $\{56,42,20;1,6,28\}$ (Section~\ref{s:56});

 \item $v=512=2^9$, $\{63,48,10;1,8,54 \}$ (Lemma~\ref{l:wd});

 \item $v=729=3^6$, $\{80,63,11;\underline{1,9},70 \}$ (Corollary~\ref{c:fact});

 \item $v=729=3^6$, $\{ 104,66,8;1,12,88 \}$; a graph does not exist \cite{Urlep:2012}.
\end{itemize}
In this note we discuss the possibility to construct
distance-regular graphs
with the parameters above as coset graphs of linear completely regular codes
in Hamming graph $H(n,q)$, where $q$ is prime.
(Note that the existence of a completely regular code in
 $H(n,q)$ where $q=p^s$ for prime $p$ implies  the existence of a completely regular code in $H(n(q-1)/(p-1),p)$ with the same intersection matrix, because  $H(n(q-1)/(p-1),p)$ covers $H(n,q)$, see also Lemma~\ref{l:11.1.10} below.)

\section{Definitions and basic facts}\label{s:def}

\subsection{Graphs}\label{ss:graphs}
All graphs in this note are finite, undirected, connected, without loops or multiple edges.
In a graph $\Gamma$,
the \emph{neighborhood} $\Gamma(\bar v)$ of a vertex $\bar v$ is the set of vertices adjacent to $\bar v$.
The {\em degree} of a vertex $\bar v$ is the size of $\Gamma(\bar v)$.
A graph is {\em  regular} if every vertex has the same degree.
The \emph{$i$-neighborhood} $\Gamma_i(\bar v)$ is the set of vertices at geodetic distance $i$ to $\bar v$.
The \emph{distance-$i$} graph of $\Gamma$ 
is the graph on the same vertex set where two vertices are adjacent
if and only if the  geodetic distance between them in $\Gamma$ is $i$.
A graph is {\em  distance-regular} (DR) if for every pair or vertices $\bar u$ and $\bar v$ at distance $i$ apart the quantities
$$
 b_i=| \Gamma_{i+1}(\bar u)\cap \Gamma(\bar v)|, \qquad
 c_i=| \Gamma_{i-1}(\bar u)\cap \Gamma(\bar v)|,
$$
which are referred to as the \emph{intersection numbers} of the graph,
solely depend on $i$ and not on the special choice of the pair $(\bar u,\bar v)$.
The {\em  automorphism group} of a graph is the set of permutations of the vertices that preserve adjacency.

Given a group $\GG$ and an inversion-closed set $G$ of non-identity elements of $\GG$,
the Cayley graph $\Cal(\GG,G)$ is the graph on the vertex set $\GG$ where two elements
$\bar u$ and $\bar v$ are adjacent if and only if $\bar u \bar v^{-1} \in G$.

The {\em Hamming graph} $H(n,q)$
 is a distance-regular graph on the set $\Sigma^n$ of $n$-words over the alphabet $\Sigma$ of size $q$,
 two vectors being adjacent
 if they are at Hamming distance one.
 We mostly consider the case when
 $q$ is a prime power,
 $\Sigma$ is the finite field $\FF_q$ of order $q$,
 and $\Sigma^n$ is
 an $n$-dimensional vector space over $\FF_q$.
The \emph{weight} of a vertex of $H(n,q)$
is the number of nonzero elements in the corresponding
tuple. The \emph{weight distribution} of a set $C$ of vertices of $H(n,q)$
is the sequence $W = (W_0,W_1,\ldots,W_d)$,
where $W_i$ is the number of words of weight $i$ in $C$.
Sometimes, we will write the weight distribution
in the form 
$\{i_0^{W_{i_0}},i_1^{W_{i_1}},\ldots,i_k^{W_{i_k}} \}$,
where $W_{i_0}$, \ldots, $W_{i_k}$ are the nonzero terms in $W$; for example, $(1,0,0,7,7,0,0,1) = \{ 0^1,3^7,4^7,7^1\}$.

 \subsection{Equitable partitions, completely regular codes}\label{ss:crc}
{ A partition $(P_0,\ldots,P_{r})$ of the vertex set of a graph $\Gamma$ is called
 an \emph{equitable partition}}
 (also known as regular partition, partition design, perfect coloring)
 if there are constants $S_{ij}$ such that
 $$|\Gamma(\bar v)\cap P_j|=S_{ij}\qquad \mbox{for every } \bar v\in P_i, \qquad i,j=0,\ldots,r.$$
 The numbers $S_{ij}$ are referred to as the intersection numbers,
 and the matrix
 $(S_{ij})_{i,j=0}^{r}$ as the \emph{intersection} (or quotient) \emph{matrix}.
 If the intersection matrix is tridiagonal,
 then $P_0$ is called a \emph{completely regular code} of covering radius $r$ and
 \emph{intersection array} $(S_{01},S_{12},\ldots, S_{r{-}1\, r}; S_{10},S_{21},\ldots, S_{r\, r{-}1})$.
 That is to say, a set of vertices $C$ is a completely regular code
 if the distance partition
 $(C=C^{(0)},C^{(1)},\ldots,C^{(r)}$) with respect to $C$ is equitable.
 As was proven by Neumaier \cite{Neumaier92}, for a distance-regular graph,
 this definition is equivalent to the original Delsarte definition \cite{Delsarte:1973}:
 a code $C$ is completely regular if
 the outer distance distribution $(|C\cap \Gamma_i(\bar v)|)_{i=0,1,2,...} $ of $C$ with respect to a vertex $\bar v$ depends only on the distance from $\bar v$ to $C$.

\subsection{Linear and additive codes}\label{ss:linear}
A  linear code (that is, a linear subspace of $\FF_q^n$) of dimension $k$ and
minimum distance $d$ in $H(n,q)$ is called an
$[n,k,d]_q$ code.
The duality is understood with respect to the standard inner product.

Linear codes are special cases of the additive codes, which are, by definition, the subgroups of the additive group of  $\FF_q^n$.
A {\em coset} of an additive code $C$ is any translate of $C$ by a constant vector. A {\em coset leader} is any coset element that minimizes the weight.
The {\em weight of a coset} is the weight of any of its leaders.
The {\em coset  graph} $\Gamma_C$ of an additive code $C$
is defined on the cosets of $C$,
two cosets being adjacent if they differ in a coset of weight one.
\begin{lemma}[see, e.g., \cite{Brouwer}]\label{l:cr-dr}
An additive code
with distance at least $3$ is completely regular
with intersection array $\{b_0,\ldots,b_{\rho-1};c_1,\ldots,c_{\rho} \}$ if and only if
the coset graph is distance-regular with intersection numbers $b_0,\ldots,b_{\rho-1}$, $c_1,\ldots,c_{\rho}$.
\end{lemma}
On the other hand,
\begin{lemma}[{see \cite[Theorem~11.1.10]{Brouwer}}]\label{l:11.1.10}
 If a distance-regular graph of diameter at least $3$
 is a Cayley graph on an elementary abelian $p$-group,
 then it is the coset graph of a
 linear completely regular $1$-code over $\FF_p$.
\end{lemma}

So, the existence of a distance-regular graph of diameter at least $3$ that is a Cayley graph on an elementary abelian $p$-group is equivalent to the existence of a linear completely regular $1$-code
in a Hamming graph over $\FF_p$ with the same intersection array.

\section{Weight distributions}\label{s:wd}
\begin{lemma}[\cite{Martin:PhD},\cite{Kro:struct}]\label{l:wd}
 If $P$ is an equitable partition of the Hamming graph $H(n,q)$
 with quotient matrix $S$,
 then $P$  is an equitable partition of the distance-$w$ graph of $H(n,q)$
 with quotient matrix $S^{(w)}:=K_w(K_1^{-1}(S))$, where
 $$K_k(x):= K_k(x;n,q):=  \sum_{j=0}^{k} (-1)^{j}(q-1)^{k-j}\binom{x}{j}\binom{n-x}{k-j} $$
 is the Krawtchouk polynomial.
 In particular, $K_w(K_1^{-1}(S))$ is integer.
 (The same is true for any distance-regular graph, with the corresponding P-po\-ly\-no\-mial.)
\end{lemma}
The sequence of matrices $(S^{(w)})_{w=0}^n$
can be called the generalized weight distribution
of the equitable partition. Indeed,
if the all-zero word belongs to $P_i$,
then the number of vertices of weight $w$ in $P_j$
is precisely the $(i,j)$th element of the matrix
$S^{(w)}$. The following corollary shows
that for some putative intersection arrays we can easily,
without any calculations, see that $S^{(w)}$
is not integer for some $w$.
\begin{corollary}\label{c:fact}
 Assume that we have a completely regular code $C$ with intersection array
$ \{\beta_0, ... , \beta_{\rho-1}; \gamma_1, ..., \gamma_{\rho}\}$
 in a distance-regular graph with  intersection array
$ \{b_0, ... , b_{D-1}; c_1, ..., c_{D}\}$.
Then   for any appropriate $i\in\{1, ..., \rho\}$ and $j\in\{0, ..., \rho-i\}$ the product
$\beta_j \beta_{j+1} ... \beta_{j+i-1}$ of
$i$ consequent intersection numbers
(similarly, for $\gamma_{j+1} ... \gamma_{j+i}$)
is divisible by $c_1c_2 ... c_i$ (for Hamming graphs, by $i!$).
\end{corollary}
\begin{proof}
We fix any vertex $v$ from $C^{(j)}$ and denote
by $W_l$ the set of vertices at distance $l$ from $v$.
We will prove by induction that
\begin{equation} \label{eq:prod}
 |W_i \cap C^{(i)}| = \prod_{l=1}^{i} \frac{\beta_{j+l-1}}{c_{l}}.
\end{equation}\label{eq:ind}
For $i=0$, it is trivially $1=1$ 
(by usual convention, the result of multiplying no factors is $1$).
For $i>0$, it is sufficient to show that
\begin{equation}
{c_{i}}|W_i \cap C^{(i)}| =
{\beta_{j+i-1}}
|W_{i-1} \cap C^{(i-1)}|
.
\end{equation}

By the definition of a completely regular code,
the number of edges beginning in $W_{i-1} \cap C^{(i-1)}$ and ending in $C^{(i-1)}$
is $\beta_{j+i-1}|W_{i-1} \cap C^{(i-1)}|$.
Since $(C^{(t)})_t$ is a distance partition, all these edges end in $W_{i}$.

By the definition of a distance-regular graph,
the number of edges beginning in $W_{i} \cap C^{(i)}$ and ending in $W_{i-1}$
is $c_{i}|W_{i} \cap C^{(i)}|$.
Since $\{C^{(t)}\}_t$ is a distance partition, all these edges end in $C^{(i-1)}$.

So, by double-counting the edges connecting $W_{i-1} \cap C^{(i-1)}$ and $W_{i} \cap C^{(i)}$ we get \eqref{eq:ind}.
By induction, we have~\eqref{eq:prod}.
Since the left part of~\eqref{eq:prod}
is integer, the claim of the corollary is now obvious.
\end{proof}

\begin{corollary}\label{c:param}
 In Hamming graphs, there are no completely regular $1$-codes with the following
 putative intersection arrays:
 
$\{22,16,5;\underline{1,2,20}\}$;
$\{25,24,3;\underline{1,3},20\}$;
$\{26,24,19;\underline{1,3},8\}$;

$\{32,28,9;\underline{1,2,28}\}$;
$\{ 36,28,4;1,2,24 \}$;
$\{ 44,36,5;\underline{1,9},40 \}$;

$\{ 48,35,9;\underline{1,7},40 \}$;
$\{63,48,10;1,8,54 \}$;
$\{80,63,11;\underline{1,9},70 \}$.
\end{corollary}
\begin{proof}
 \newcommand\lB{\linebreak[2]}
 For all arrays except $\{ 36,\lB 28,\lB 4;\lB 1,\lB 2,\lB 24 \}$ and $\{63,\lB 48,\lB 10;\lB 1,\lB 8,\lB 54 \}$,
 the existence of completely regular codes contradicts Corollary~\ref{c:fact},
 which states that 
 $\gamma_1\gamma_2$ is divisible by $2$
 (false for
$\{25,\lB 24,\lB 3;\lB \underline{1,\lB 3},\lB 20\}$,
$\{26,\lB 24,\lB 19;\lB \underline{1,\lB 3},\lB 8\}$,
$\{ 44, \lB 36, \lB 5;\lB \underline{1,\lB 9},\lB 40 \}$,
$\{ 48,\lB 35,\lB 9;\lB \underline{1,\lB 7},\lB 40 \}$, and
$\{80,\lB 63,\lB 11;\lB \underline{1,\lB 9},\lB 70 \}$)
and $\gamma_1\gamma_2\gamma_3$ is divisible by $6$
(false for $\{22,\lB 16,\lB 5;\lB \underline{1,\lB 2,\lB 20}\}$ and 
$\{32,\lB 28,\lB 9;\lB \underline{1,\lB 2,\lB 28}\}$).

 For $\{ 36,28,4;1,2,24 \}$, since $36$ is the degree $(q-1)n$ of the Hamming graph
 $H(n,q)$, $n\ge 3$, we have $q\in\{2,\lB 3,\lB 4,\lB 5,\lB 7,\lB 10,\lB 13\}$. In all the cases,
 direct calculations show that $K_3(K_1^{-1}(S))$ is not integer,
 where $K_w(\cdot)=K_w(\cdot;n,q)$ and $S$ is the
 quotient matrix corresponding to the array $\{ 36,\lB 28,\lB 4;\lB 1,\lB 2,\lB 24 \}$.

 For $\{63,48,10;1,8,54 \}$, the proof is similar.
\end{proof}

\section[\{42,30,12;1,6,28\}, q=7]{\{42,30,12;1,6,28\}, \boldmath $q=7$} \label{s:42}

A putative distance-regular graph with intersection array $\{42,30,12;1,6,28\}$
has order $343=7^3$. If such a graph can be realized as the coset graph
of a linear completely regular code in a Hamming graph,
then it is a $7$-ary code of length $7=42/(6-1)$, dimension $4=7-3$,
distance $3$, and covering radius $3$.
The weight distributions of the code and of its dual are
$( 1, 0, 0, 42, 42, 630, 840, 846)$ and $(1, 0, 0, 0, 42, 0, 210, 90)$, respectively.

\begin{proposition}\label{p:7}
If $q$ is an odd prime power larger than $4$,
then there are no $q$-ary codes of length $7$,
dimension $3$,
and non-zero weights $4$, $6$, $7$.
\end{proposition}
\begin{proof}
 Assume that such a code $C$ exists. W.l.o.g., it has a generator matrix
 $$
 \left(
 \begin{array}{cccccccc}
 1 & 0 & 0 & ? & ? & ? & ? \\
 0 & 1 & 0 & ? & ? & ? & ? \\
 0 & 0 & 1 & ? & ? & ? & ?
 \end{array}
 \right)
 $$
 We note that the weight of each row can only be $4$,
 and two rows do not have two common zero positions
 (otherwise, for $q>4$, they can be combined in a weight-$5$ codeword
 contradicting the hypothesis).
 So, w.l.o.g., we have
  $$
 \left(
 \begin{array}{cccccccc}
 1 & 0 & 0 & 0 & 1 & 1 & 1 \\
 0 & 1 & 0 & a & 0 & b & c \\
 0 & 0 & 1 & x & y & 0 & z
 \end{array}
 \right).
 $$
 If $b\ne c$, then there is a weight-$5$
 linear combination of the first two rows, a contradiction.
 So, $b=c$; similarly, $y=z$; and, if we assume w.l.o.g.
 that $a=b$, we also see $x=y$:
   $$
 \left(
 \begin{array}{cccccccc}
 1 & 0 & 0 & 0 & 1 & 1 & 1 \\
 0 & 1 & 0 & a & 0 & a & a \\
 0 & 0 & 1 & x & x & 0 & x
 \end{array}
 \right).
 $$
Multiplying rows by coefficients, we get
$$
 \left(
 \begin{array}{cccccccc}
 -1 & 0 & 0 & 0 & -1 & -1 & -1 \\
 0 & a^{-1} & 0 & 1 & 0 & 1 & 1 \\
 0 & 0 & x^{-1} & 1 & 1 & 0 & 1
 \end{array}
 \right).
 $$
Now, the sum of the rows $(-1,a^{-1},x^{-1},2,0,0,1)$
has weight $5$, a contradiction.
\end{proof}

\begin{corollary}\label{c:7}
 There are no linear completely regular codes with intersection array
 $\{42,30,12;1,6,28\}$ in $H(7,7)$.
\end{corollary}

\section[\{52,42,16;1,6,28\}, q=5]{\{52,42,16;1,6,28\}, \boldmath $q=5$}\label{s:52}

A putative distance-regular graph with intersection array $\{52,42,16;1,6,28\}$
has order $625=5^4$. If such a graph can be realized as the coset graph
of a linear completely regular code in a Hamming graph,
then it is a $5$-ary code of length $13=52/(5-1)$, dimension $9=13-4$,
distance $3$, and covering radius $3$.
The weight distributions of the code and of its dual are
$(1,$ $0,$ $0,$ $52,$ $260,$ $2028,$ $10660,$ $48620,$ $128076,$
$305240,$ $479596,$ $521040,$ $350480,$ $107072)$
 and $(1,0,0,0,0,0,0,52,0,0,364,0,208,0)=\{0^1,7^{52},10^{364},12^{208}\}$, respectively.

\begin{proposition}\label{p:5}
There is no $5$-ary code of length $13$, dimension $4$, and non-zero weights $7$, $10$, $12$.
\end{proposition}
\begin{proof}
 Let $G$ be a $4\times 13$ generator matrix of a putative code with considered parameters.
 The columns of $G$ form a set $S$ of $13$ points of the vector space $\mathbb{F}_5^4 = \mathrm{GF}(5)^{4}$.
 Any ($3$-dimensional) hyperplane $H\ni \bar 0$ in $\mathbb{F}_5^4$ is determined by its dual vector
 $\bar h$: $H=\{\bar x \in \mathbb{F}_5^4 : \bar x\cdot \bar h=0\}$.
 Since $\bar h \cdot G$ is a codeword of weight $7$, $10$, or $12$,
 any such hyperplane intersects with $S$ in $6=13-7$, $3=13-10$, or $1=13-12$ points.
 Now, consider a ($2$-dimensional) plane $P$ containing at least two points of $S$.
 It is included in exactly six 
 $3$-dimensional hyperplanes, call them $H_i$, $i=0,1,2,3,4,5$.
 In the rest of the proof, we will show that any intersection numbers
 of $S$ with $P$ and $H_i$, $i=0,1,2,3,4,5$, are contradictory.
 Denote $a:= |S \cap P|$, $a>1$, and $b_i :=|S \cap H_i|$, $b_i\in\{3,6\}$.
 Now we have
 $13=|S| = a + \sum_{i=0}^5 (b_i-a) = \sum_{i=0}^5 b_i-5a = 6\cdot 3 + k\cdot 3 - 5a$,
 where $k$ is the number of $i$ in $\{0,1,2,3,4,5\}$ such that $b_i=6$.
 We derive $k=5(a-1)/3$. Since $k$ is integer, we have $a=4$ and $k=5$.
 However, $a=4$ also implies that $b_i$ cannot be $3$ and hence $k=6$, a contradiction.
\end{proof}

\begin{corollary}\label{c:5}
 There are no linear completely regular codes with intersection array
 $\{52,42,16;1,6,28\}$ in $H(13,5)$.
\end{corollary}

\section[\{56,42,20;1,6,28\}, q=3]{\{56,42,20;1,6,28\}, \boldmath $q=3$}\label{s:56}

A putative distance-regular graph with intersection array $\{56,42,20;1,6,28\}$
has order $729=3^6$. If such a graph can be realized as the coset graph
of a linear completely regular code in a Hamming graph,
then it is a $3$-ary code of length $28=56/(3-1)$, dimension $22=28-6$,
distance $3$, and covering radius $3$.
The weight distributions of the dual code is
$\{0^1, 12^{56}, 18^{392}, 21^{280}\}$.

\begin{proposition}\label{p:3}
There is no $3$-ary code of length $28$, dimension $6$, and non-zero weights $12$, $18$, $21$.
\end{proposition}
We currently do not have a theoretical proof of the last proposition.
 The nonexistence of a code with the specified parameters was checked
 using the software~\cite{Kurz:LINCODE} (1 hour of computation)
 and~\cite{QextNewEdition} (10 sec. of computation).

\begin{corollary}\label{c:3}
 There are no linear completely regular codes with intersection array
 $\{56,42,20;1,6,28\}$ in $H(28,3)$.
\end{corollary}
The following straightforward corollary of the fact above
has an easy computer-free proof.
\begin{corollary}\label{c:9}
 There are no linear completely regular codes with intersection array
 $\{56,42,20;1,6,28\}$ in $H(7,9)$.
\end{corollary}
\begin{proof}{}
 If such a code exists, then its dual weight distribution
 is $\{0^1,$ $4^{56},$ $6^{392},$ $7^{280}\},$
 which is impossible by
 Proposition~\ref{p:7}.
\end{proof}

\section[\{140,126,20,1;1,10,126,140\}, q=3]{\{140,126,20,1;1,10,126,140\}, \boldmath $q=3$} \label{s:140}

The only putative distance-regular graph in the [BCN] table
with prime-power order and diameter
more than $3$ has order $2187=3^7$
and
intersection array
$\{140,126,20,1;1,10,126,140\}$.

The corresponding completely regular code $C$ in $H(70,3)$
has intersection matrix
$$\left(
\begin{array}{ccccc}
  0 & 140 &   0 &   0 &  0 \\
  1 &  13 & 126 &   0 &  0 \\
  0 &  10 & 110 &  20 &  0 \\
  0 &   0 & 126 &  13 &  1 \\
  0 &   0 &   0 & 140 &  0
\end{array}\right).
$$
It is easy to see that $C \cup C^{(4)}$ is a completely regular code with intersection matrix
$$\left(
\begin{array}{ccccc}
  0 & 140 &  0  \\
  1 & 13  & 126 \\
  0 & 30 & 110
\end{array}\right)
.$$
Since $C^{(4)}$ consists of $3$ cosets of $C$, the code $C \cup C^{(4)}$ is linear;
its coset graph is a strongly regular graph
whose parameters are also questionable,
according to~\cite{Brouwer:param}.


The dual code is a ternary two-weight code
with weight distribution  $\{ 0^{1}, 45^{588}, 54^{140} \}$.
However, a $[70,6,45]_3$ code does not exist \cite{Maruta:2004}.

\begin{corollary}\label{c:3-70}
 There are no linear completely regular codes with intersection array
 $\{140,126,20,1;1,10,126,140\}$ or $\{140,126;1,30\}$ in $H(70,3)$.
\end{corollary}

\section{Conclusion}\label{s:concl}
We have shown that
no distance-regular graphs
with an unsolved intersection array from
the table~\cite[Ch.\,14]{Brouwer}
can be constructed as a Cayley graph
on an elementary abelian group,
or, equivalently, as the coset graph of a linear completely regular code.

Finally, we observe that the existence of unrestricted (not necessarily linear) completely regular codes with the  parameters considered in Sections~\ref{s:42}--\ref{s:140}
and the existence of few-distance codes with dual parameters remain unsolved.

\subsection*{Acknowledgments}
The authors thank Sasha Kurz for helpful discussions and computer help.


\providecommand\href[2]{#2} \providecommand\url[1]{\href{#1}{#1}}
  \def\DOI#1{{\small {DOI}:
  \href{http://dx.doi.org/#1}{#1}}}\def\DOIURL#1#2{{\small{DOI}:
  \href{http://dx.doi.org/#2}{#1}}}

\end{document}